\documentclass{amsart}
\UseRawInputEncoding

\usepackage{bm}
\usepackage{amsmath,amsfonts,amssymb}
\usepackage{graphicx,psfrag}

\newtheorem{teo}{Theorem}

\newtheorem{prop}{Proposition}

\newtheorem{lemma}{Lemma}

\newcommand{\R}{{\mathbb{R}}}

\newcommand{\C}{{\mathbb{C}}}
\newcommand{\Z}{{\mathbb{Z}}}
\newcommand{\N}{{\mathbb{N}}}
\newcommand{\T}{{\mathbb{T}}}

\newcommand{\Of}{\mathcal{O}_f}

\begin{document}
\title{The growth rate inequality for Thurston maps with non hyperbolic orbifolds.}
\author{J.Iglesias, A.Portela, A.Rovella and J.Xavier}

\begin{abstract} Let $f: S^2 \to S^2$ be a continuous map of degree $d$, $|d|>1$, and let $N_nf$ denote the number of fixed points 
of $f^n$. We show that if $f$
is a  Thurston map with non hyperbolic orbifold, then either the growth rate inequality $\limsup \frac{1}{n} \log N_nf\geq \log |d|$ holds for $f$ or 
$f$ has exactly two  critical points which are fixed and totally invariant.

\end{abstract}

\maketitle

\section{Introduction}

For some time we have been interested in the following open problem:  let $f: S^2 \to S^2$ be a continuous map of degree $d$, $|d|>1$, and let $N_nf$ denote the number of fixed points 
of $f^n$. When does
the growth rate inequality $\limsup \frac{1}{n} \log N_nf\geq \log |d|$ hold for $f$? (This is Problem 3 posed in \cite{shub2}).  There are many recent papers written on the subject,
where hypothesis on the set of recurrent points are imposed in order to guarantee the desired growth rate. Note that in the total absence of recurrence, the inequality is easily
seen to fail, as the extension to $S^2$ fixing infinity of the planar map in polar coordinates $(r,\theta)\mapsto (2r, 2\theta)$ shows. Also the existence of certain invariant foliations
has been shown to guarantee the growth rate.\\

In this work we take another direction and focus  on the structure of the postcritical set: are there hypothesis on this set that guarantee the desired growth rate? We show that
for Thurston maps with non hyperbolic orbifolds either the growth rate is satisfied, or the structure of the postcritical set is exactly as in the counterexample described above:
there are exactly two critical points which are fixed and totally invariant.  See the next section for the definitions.\\

Thurston maps with non hyperbolic orbifolds were completely classified in \cite{dh}. It turns out that these maps can be lifted to degree $d$ covering maps $(|d|>1)$ of either the open annulus or the torus, or even act directly on the open annulus by restriction (for example, when
there are exactly two totally invariant critical points as in the map $(r,\theta)\mapsto (2r, 2\theta)$ described above).  The theory of torus coverings is of course well understood,
even from the point of view of the growth rate inequality, because of compactness and homological considerations.  The theory of self coverings of the open annulus was developed
by the authors, very much focusing on the growth rate inequality. The lack of compactness is of course an issue here. As a consequence, these very special maps (Thurston maps with
non hyperbolic orbifolds) are a natural enviroment to tackle this elusive growth rate problem.\\

Some of the results presented here are not exactly new.  For example, it was shown in \cite{Li} that each {\it expanding} Thurston map has $\deg f +1$ fixed points, counted with 
appropriate weight.  The definition of an expanding Thurston map is rather technical and we will not discuss it here.  Nevertheless, we point out that we do not assume our maps to be expanding, but
more importantly, even when they are expanding, our proofs are entirely different.  We give purely topological proofs of the results, using only elementry algebraic topology and 
the Lefschetz fixed point theorem.

\section{Thurston maps with non hyperbolic orbifolds.}

If $f:S^2\to S^2$ is a branched covering, we denote by $\deg_xf$ the local 
degree of $f$ at $x$.  We will call $S_f=\{x|\deg_xf>1\}$ the critical set of $f$, and $P_f=\cup_{n>0}f^n(S_f)$ the postcritical set.\\

 A Thurston map is an orientation preserving branched covering of the 
sphere onto itself such that the postcritical set $P_f$ is finite. The ramification function $\nu_f$  of a Thurston map $f$ is 
the smallest among functions $\nu:S^2\to \N^*\cup\{\infty\}$ such that \begin{itemize}
                                                                         \item $\nu(x)=1$ if $x\notin P_f$
                                                                         \item $\nu(x)$ is a multiple of $\nu(y)\deg_y(f)$ for each $y\in f^{-1}(x)$.\\

                                                                        \end{itemize}
                                                                        
The orbifold associated to a Thurston map $f$ is the pair $(\mathcal{O}_f, \nu_f)$. Note that $\{p\in S^2:\nu_f(p) \geq 2\}=P_f$ is a finite set.  If we label these points
$p_1,\dots,p_n$ such that $2\leq\nu_f(p_1)\leq\ldots \leq \nu_f(p_n)$, then the $n$-tuple $(\nu_f(p_1), \ldots, \nu_f(p_n))$ is called the signature of $(\mathcal{O}_f, \nu_f)$.
Thus the numbers appearing in the $n$-tuples are the 
values of the ramification
function, which assign ``weights'' to points in $P_f$.\\

The Euler characteristic of $(\mathcal{O}_f, \nu_f)$ is 
$$\chi (O_f)=2-\sum_{x\in P_f} (1-\frac{1}{\nu_f(x)}).$$ 

A Thurston map has a non hyperbolic orbifold if $\chi (O_f)=0$. A Thurston map has a non hyperbolic orbifold if and only if $\deg_pf . \nu_f(p) = \nu_f(f(p))$ for all $p\in S^2$
and if and only if the signature of $\Of$ is   $(\infty, \infty)$, $(2,2,\infty)$, $(2,4,4)$, $(2,3,6)$, $(3,3,3)$ or
$(2,2,2,2)$ (see, for example \cite{bm} Proposition 2.14).  If there is no place to confusion,  will often write $\nu$ instead of $\nu_f$ for the ramification function.  \\

Recall that $N_nf$ denote the number of fixed points 
of $f^n$. We will prove the following:

\begin{teo} Let $f: S^2 \to S^2$ be a Thurston map with non hyperbolic orbifold and degree $d$, $|d|>1$. Then either the growth rate inequality 
$\limsup \frac{1}{n} \log N_nf\geq \log |d|$ holds for $f$ or 
$f$ has exactly two critical points which are fixed and totally invariant.
 
\end{teo}

It follows that a Thurston map with non hyperbolic orbifold either satisfies the growth rate inequality or  the signature of $f$ is $(\infty, \infty)$ and
$P_f=S_f=\{p,q\}\subset N_1f$.\\

From now on, whenever $f:S^2\to S^2$ satisfies the growth rate inequality, we will say that $f$ ``has the rate''.\\

\section{Further definitions, notations and preliminaries}

If $f:X\to Y$, $x\in X, y\in Y$, the number $\delta_{xy}$ is $1$ if $f(x)=y$ and $0$ otherwise. We  write $f:(X,x)\to (Y,y)$ meaning that $x\in X, y\in Y$ and $\delta_{xy}=1$.\\

If $f:(X,x_0)\to (Y,y_0)$ is continuous, we write $f_*:\pi_1(X,x_0)\to \pi_1 (Y,y_0)$ and $f_{\#}:H_1(X,\Z)\to H_1 (Y,\Z)$ for the 
corresponding actions in fundamental 
groups and first homology groups respectively. We write $[\gamma]\in H_1(X,\Z)$ for $\gamma \in \pi_1(X,x_0)$.  So, $[f_*(\gamma)]=f_{\#}([\gamma])$. \\

We adopt the convention that maps between spaces are always assumed to be continuous unless otherwise stated.\\

We will often come across the topological space $S^2\backslash \{x_1, \ldots,x_n\}$.  We refer to the points $x_1, \ldots,x_n$ as the ``punctures''. In this situation, 
we consider pairwise disjoint closed topological disks  $D_i\subset S^2$ such that $x_i\in D_i$, $i=1, \ldots, n$.  We denote $\sigma_{x_i}$ the boundary of $D_i$ with
the positive orientation, $i=1, \ldots, n$.\\

We denote by $\left<g_{\alpha}, \alpha \in \Lambda\right>$ the free group in generators $g_{\alpha}, \alpha \in \Lambda$.  So, in particular taking 
$X=S^2\backslash \{x_1, \ldots,x_n\}$, we have that $\pi_1(X,x_0)=\left<\gamma_1,\ldots, \gamma_{n-1}\right>$, where each $\gamma_i$ is the homotopy class of a loop based at $x_0$ 
that is freely homotopic
in $X$ to $\sigma_{x_i}$, $i=1,\ldots,n-1 $. We also point out that the homology class $[\sigma_{x_n}]\in H_1(X,\Z)= -([\gamma_1] + \ldots+ [\gamma_{n-1}])$. We make the identification 
$H_1(X,\Z)\simeq \Z[\gamma_1]\oplus \ldots \oplus \Z[\gamma_{n-1}]$, and write
$[\gamma]=k_1[\gamma_1]+ \ldots+k_{n-1}[\gamma_{n-1}] \in H_1(X,\Z),  k_i \in \Z, i=1,\ldots,n-1$, for $\gamma\in \pi_1(X,x_0) $.\\

If $\pi: (\tilde X,\tilde x_0) \to (X,x_0)$ is a covering space and $f:Y\to X$ is a map, a lift of $f$ is a map $\tilde f:Y\to \tilde X$ such that $\pi\tilde f=f$.  In the case
that such a lift exists, we say that the map $f$ lifts to $\tilde X$. We recall here the lifting criterion (Proposition 1.33, page 61 in \cite{hatcher}), which we will use repeatedly:\\ 

\begin{prop}\label{lift} Let  $\pi: (\tilde X,\tilde x_0) \to (X,x_0)$ be a covering space and $f:(Y,y_0)\to (X,x_0)$ a  map with $Y$ path-connected and locally
path-connected.  Then a lift $f:(Y,y_0)\to (\tilde X,\tilde x_0)$ of $f$ exists iff 
$f_*(\pi_1 (Y, y_0))\subset \pi_* (\pi_1 (\tilde X,\tilde x_0))$.\\
\end{prop}

\section{$(\bm{\infty},\bm{\infty})$}
In this case, we have $P_f=S_f=\{p,q\}$ and  the set $\{p,q\}$ is totally invariant. Then, $f$ acts by restriction on the open
annulus $S^2 \backslash \{p,q\}$. As was pointed out in the Introduction, the growth rate is not necessarily satisfied in this class. We refer the reader to \cite{iprx1}, 
\cite{iprx2},\cite{iprx3} for a thorough study of the growth rate inequality problem for annulus maps.  For example, we prove that if $f$ is a self-map of the open annulus $A$ 
such that the ends of $A$ are both repelling or both attracting and $|\deg(f)|>1$, then $f$ has the rate.  In particular, if $f$ is $C^1$ and $|\deg(f)|>1$, then $f$ has the rate (both
ends are attracting in the $C^1$ setting).\\

Relative to the present work, we make the following remark here that will be used later: if a map
in this class does not  satisfy the growth rate, then necesarilly $p$ and $q$ are fixed (and not a period 2 critical cycle). 
This is Remark 4. (2) in \cite{iprx3}:

\begin{teo}\label{sarkoski} Let $F$ be a map of the open annulus $A$ that interchanges the ends of $A$. If $|\deg(F)|>1$, then $F$ has the rate.

 \end{teo}
 
The proof of this theorem can be quickly summed up for the experienced reader:  Lefschetz fixed point theorem combined with Nielsen theory
in the universal covering space. For details of the proof see the reference above.\\

\section{$(\bm{2},\bm{2},\bm{\infty})$}

Let $f:S^2\to S^2$ be a Thurston map in the class $(2,2,\infty)$. The equation $\deg_xf . \nu_f(x) = \nu_f(f(x))$ for all $x\in S^2$ implies that $f$ is a Thurston polynomial, 
that is, there exists a totally invariant fixed point (the one 
point in $P_f$ with weight $\infty$). Moreover,  the points $p$ and $q$ with weight $2$ are regular points. Deleting the totally invariant critical point, we obtain a branched 
covering $f:\C\to \C$, with postcritical set $\{p,q\}$ disjoint from the critical points.  \\

Let $\Gamma$ be the subgroup of  automorphisms of the complex plane $\C$ generated by $z\mapsto z+1$ and $z\mapsto -z$.  The quotient space $\C/\Gamma$ is topologically a plane, 
and the quotient projection $\pi: \C/<z\mapsto z+1>\to \C/\Gamma$ is a two-fold branched covering of the open annulus $A=\C/<z\mapsto z+1>$ onto the plane with two branched points
$b_1$ and $b_2$. 
 We may 
as well assume that $\pi(b_1)= q$ and $\pi(b_2)=p$. Then, $\pi: \tilde X = A\backslash \{b_1, b_2\} \to \C\backslash \{p,q\} = X$ is a covering map.  Let  
$Y=X\backslash f^{-1}(\{p,q\})$ and $\tilde Y=A\backslash \pi^{-1}(f^{-1}(\{p,q\}))$. \\

\begin{lemma} The map $f\pi|_{\tilde Y}:\tilde Y\to X$ lifts to $\tilde X$.
 
\end{lemma}

\begin{proof} We want to apply Proposition \ref{lift} to the map $f\pi|_{\tilde Y}:(\tilde Y, \tilde y_0)\to (X,x_0)$, 
and the covering space $\pi: (\tilde X,\tilde x_0) \to (X,x_0)$.  We need then to 
ensure that 
$f_*{\pi|_{\tilde Y}}_* (\pi_1 (\tilde Y,\tilde y_0))\subset \pi_* (\pi_1 (\tilde X,\tilde x_0)) =G$.  Note that $\pi_1(X,x_0)=\left<a,b\right>$, where $a$ is the homotopy class of a 
loop based at $x_0$ freely homotopic to $\sigma_p$, and $b$ is the homotopy class of a loop based at $x_0$
freely homotopic to $\sigma_q$. Moreover,  $G= \left<a^2, b^2, ab\right>$.  
We make the identification 
$H_1(X,\Z)\simeq \Z [a]\oplus \Z[b]$, and write $[\gamma]=m[a] + n [b] \in H_1(X,\Z), m,n \in \Z$ for $\gamma\in \pi_1(X,x_0) $. Note also 
that
$\gamma \in G$ if and only if $[\gamma]$ verifies $m+n =0 \mod 2$.\\

Take $y_0=\pi(\tilde y_0)$ and let $\left<\alpha, \beta,\gamma_1, \ldots, \gamma_n\right>$ generate $\pi_1(Y,y_0)$,  where the generators correspond to the punctures at $p,q$ and
their $n$ critical preimages 
outside $\{p,q\}$. That is, each homotopy class of a
loop in the generating set is freely homotopic to $\sigma_x$ for some  $x\in f^{-1}(\{p,q\})=\{p,q,c_i, i=1,\ldots,n\}$. \\

Then, 
$$H={\pi|_{\tilde Y}}_*(\pi_1 (\tilde Y,\tilde y_0))=\left < \alpha^2, \beta^2,\alpha\beta,\gamma_i,\alpha \gamma_i\alpha^{-1}, i=1,\ldots, n\right>.$$\\
We  write
$[\gamma]=m[\alpha]+n[\beta]+ k_1[\gamma_1]+ \ldots+k_n[\gamma_n] \in H_1(Y,\Z), m,n, k_i \in \Z, i=1,\ldots,n$ for $\gamma\in \pi_1(Y,y_0) $. Note  
that
$\gamma \in H$ if and only if $[\gamma]$ verifies $m+n =0 \mod 2$.\\

We have illustrated the covering spaces $\tilde Y$ and $\tilde X$ modulo homotopy equivalence in Figure \ref{228} (this picture is explained in detail at the end of the next section).\\

\begin{figure}[h]
 {\includegraphics[scale=0.5]{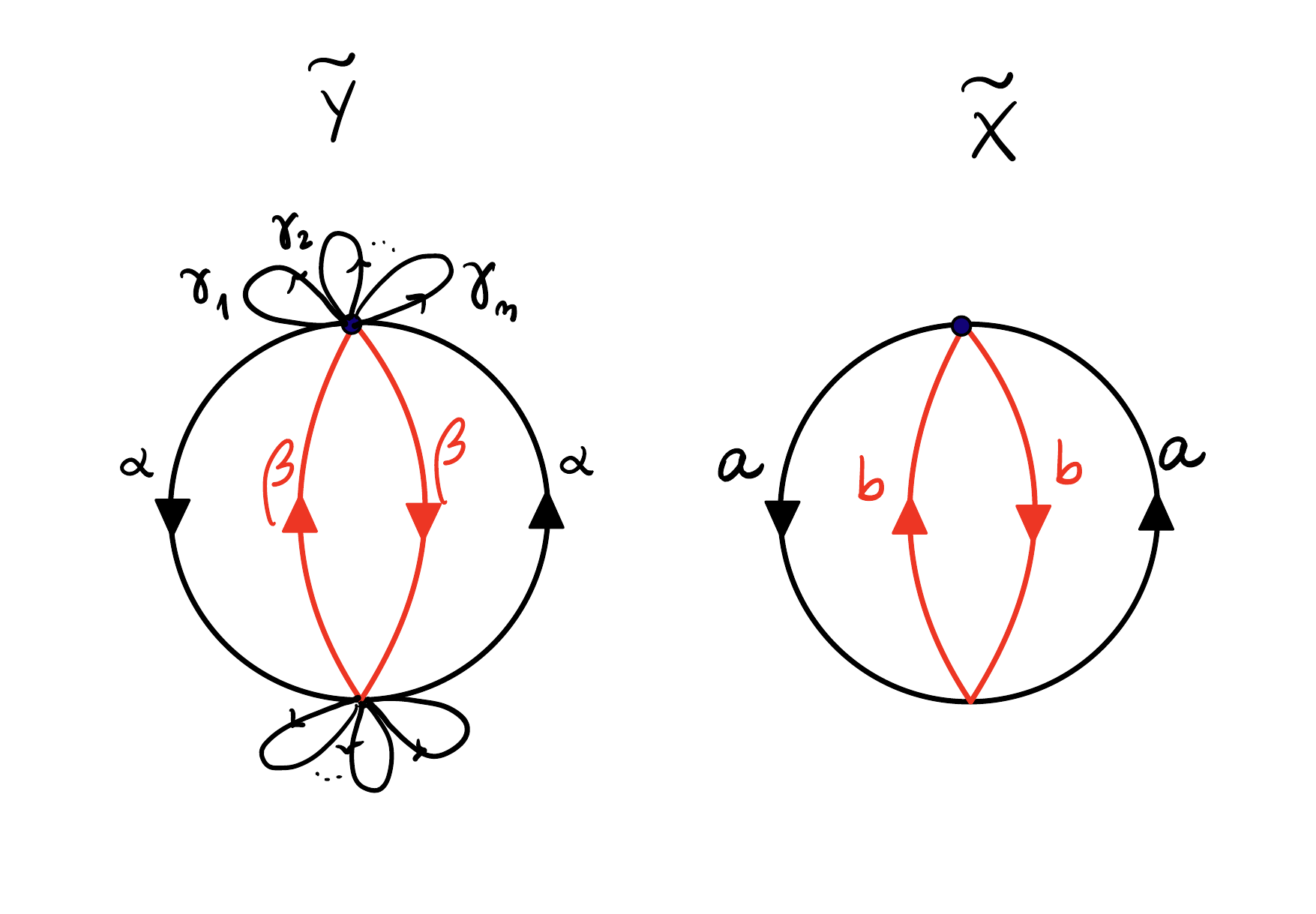}}
\caption{\label{228} $H=\left < \alpha^2, \beta^2,\alpha\beta,\gamma_i,\alpha \gamma_i\alpha^{-1}, i=1,\ldots, n\right>$; $G= \left<a^2, b^2, ab\right>$}
\end{figure}

As both $p$ and $q$ are regular 
points and the set $\{p,q\}$ is invariant, 
we have that $f_{\#}([\alpha])$ and $f_{\#}([\beta])$ are  either $[a]$ or $[b]$. Note also that the equation $\deg_xf . \nu(x) = \nu(f(x))$ for all $x\in S^2$ implies that 
 $\deg_{c_i}f = 2$ for all $i=1,\ldots,n$, so $f_{\#}([\gamma_i])$ is  either $2[a]$ or $2[b]$.  Then, 
 $$f_{\#}(m[\alpha]+n[\beta]+ k_1[\gamma_1]+ \ldots+k_n[\gamma_n])=(x+\sum_{j:f(c_j)=a} 2k_j)[a] + (y+\sum_{j:f(c_j)=b} 2k_j)[b],$$\noindent
 where $x+y=m+n$.  It follows that $x+\sum_{j:f(c_j)=a} 2k_j+ y+\sum_{j:f(c_j)=b} 2k_j= m+n \mod 2$.  So, $f_{*}(\gamma)\in G$ if $\gamma \in H$.\\

\end{proof}

We let $\tilde f:\tilde Y\to \tilde X$ be the lift given by the previous lemma.  That is, $\pi\tilde f=f\pi|_{\tilde Y}$.

\begin{lemma} The lift $\tilde f:\tilde Y\to \tilde X$ extends continuously to  $F: A \to A$ and $\deg(f) = \deg (F)$.
 
\end{lemma}

\begin{proof} The lift $\tilde f$ is defined on  $\tilde Y=A\backslash \pi^{-1}(f^{-1}(\{p,q\}))$, so it is enough to define $\tilde f$ on the set $ \pi^{-1}(f^{-1}(\{p,q\})$. 
Note now that $p$ and $q$ are the critical values of $\pi$, that is, $p$ and $q$ are the only points in $S^2$ that have only one preimage by $\pi$, namely the points $b_1$ and
$b_2 \in A$. So, if $ x\in \pi^{-1}(f^{-1}(\{p,q\})$, then $f\pi(x)$ is either $p$ or $q$ and so $f\pi(x)$ has a uniquely defined preimage by $\pi$. We can now define 
$F(x) = \pi^{-1}f\pi(x)$ if $x\notin Y$ and $F(x)= \tilde f(x)$ if $x\in Y$.  The resulting map $F$ is  continuous by construction.
To show that  $\deg(f) = \deg (F)$, let $d=\deg(f)$. Note that $f$ is locally $z\to z^d$ around $\infty$ and that $\pi$ is a homeomorphism in a neighbourhood of any of the two preimages of $\infty$.
So, we can take a simple loop $\gamma\subset S^2$ in a suitable neighbourhood of $\infty$,  take a simple loop  $\tilde \gamma\in \pi^{-1}(\gamma)$ and note that its homology class $[\tilde \gamma]$ 
generates $H_1(A,\Z)$. Then, we have  $[F(\tilde \gamma)]=d [\tilde\gamma]$, that is, $ \deg (F)=d$ .

\end{proof}

As a consequence, we obtain the following result:

\begin{teo}\label{22infty} Let $f:S^2\to S^2$ be a Thurston map with $(2,2,\infty)$ orbifold.  Then $f$ has the rate.
 
\end{teo}

\begin{proof} We first show that the  map $F:A\to A$ given by the previous lemma can be chosen so that it exchanges the ends of $A$. Let 
$T: A\backslash \{b_1, b_2\} \to A\backslash \{b_1, b_2\}$ be a non-trivial covering transformation (that is, $\pi=\pi T$).  Note that if the map $F:A\to A$ 
fixes the ends of $A$, then $TF$ exchanges them.  Now,  Theorem \ref{sarkoski} implies that either $F$ or $TF$ has the rate, but this implies that $f$ has the rate as the
covering projection $\pi$
is two-fold (for every iterate $k$ the fixed points of $f^k$ are at least half of those of $F^k$).
 
\end{proof}

We point out here that  no Thurston polynomial is expanding (see for example Lemma 6.8 in \cite{bm}). So, the rate for maps with signature $(2,2,\infty)$ does not follow from the results in \cite{Li}.\\

We finish this section relating Theorem \ref{22infty} with Garc\'ia's paper \cite{alejo}. In that work the author gives sufficient conditions for degree $2$ branched covering maps of
the plane to have a fixed point. For example, if $f:\R^2\to \R^2$ is a degree $2$ branched covering and $K\subset \R^2$ a totally invarint continuum that does not separate the critical point
from its image, then $f$ has a fixed point. In the very special case that the  critical point $c$ of such a covering $f$ satisfies $c\neq f(c)$, $f^2 (c) = p$ and $p$ is a fixed point,
we can use our previous result to obtain the much stronger
conclusion that $f$ has the rate (just note that $f$ extends to $S^2$ in the $(2,2,\infty)$ class).  In this case, there are no additional dynamical hypothesis, only those imposed on
the postcritical set. \\

\section{Branched coverings from tori}

The remaining cases have the property of lifting to  certain branched coverings $\pi:T^2 \to S^2$.  The following lemmas will come in handy:

\begin{lemma}\label{vap1} Let $T: \R^2\to \R^2$ be $T(z)=e^{i\theta} z$, $\theta \neq 2k\pi,k\in \Z$, and let $p: \R^2 \to  \T^2$ be a covering projection
 such that $pT=\varphi p$, with  $\varphi:\T^2\to \T^2$.  Let $F: \T^2  \to \T^2$ be a degree $d$ map with $|d|>1$ and let  $\tilde F:\R^2 \to \R^2$ such that
 $p\tilde F = Fp$.  If $F_*$ has $1$ as an eigenvalue, then $\varphi_*F_*$
does not. 
 
\end{lemma}

\begin{proof} If $F_*$ has $1$ as an eigenvalue, then the other eigenvalue is $d$. Let $\{v_1,v_2\}$ be a basis of $\R^2$ such that $F_*(v_1)=v_1$ and $F_*(v_2)=dv_2$.  Assume that
there exists $v \in \R^2$ such that $\varphi_*F_*(v)=v$ (we want to show that $v=0$).  We write $v=av_1+bv_2$ and note that $F_*(v)=av_1 +b d v_2$.  Also, $\varphi_*(v)= e^{i\theta}v$
for all $v\in \R^2$.  Then, $\varphi_*F_*(v)=\varphi_*(av_1 +b d v_2)=a e^{i\theta}v_1 +db e^{i\theta}v_2$.  Solving for $\varphi_*F_*(v)=v$ we get $a=a e^{i\theta}$ and 
$b=db e^{i\theta}$.  From the first equation we get that etiher $a=0$ or $e^{i\theta} =1$.  Using that $\theta \neq 2k\pi,k\in \Z$, we get $a=0$.  From the second equation 
we get that either $b=0$ or $de^{i\theta}=1$.  Using $|d|>1$ we get $b=0$.
 
\end{proof}

\begin{lemma}\label{accion} Let $T: \R^2\to \R^2$ be $T(z)=e^{i\theta} z$, $\theta \neq 2k\pi,k\in \Z$, and let $p: \R^2 \to  \T^2$ be a covering projection
 such that $pT=\varphi p$, with  $\varphi:\T^2\to \T^2$. Let $\pi:\T^2 \to \T^2/\sim$ be the quotient map by the action of $\varphi$, and suppose that $\pi$
 is a branched covering $\pi:\T^2 \to S^2$. Let $f: S^2\to S^2$ such that there exists $F: \T^2  \to \T^2$ with $f\pi=\pi F$. Then, $f$ has the rate.
 
\end{lemma}

\begin{proof} It follows  from Theorem 1.2 page 618 of \cite{bfgj} that if $F_*:H_1(\T,\Z)\to H_1(\T,\Z)$ does not have $1$ as an eigenvalue, then
$F$ satisfies the growth rate inequality $\limsup \frac{1}{n} \log N_nF\geq \log |d|$, where $N_nF$ denote the number of fixed points of $F^n$. Note now that the
covering projection $\pi$ has a finite number of sheets $l$. This implies that $f$ has the rate as  for every iterate $n$ the fixed points of $f^n$ are at least $1/l$  those of $F^n$.
Let us suppose now that $F_*:H_1(\T,\Z)\to H_1(\T,\Z)$ has $1$ as an eigenvalue.  By Lemma \ref{vap1} $\varphi_*F_*$ does not. Moreover, $\varphi F$ is also a  lift of $f$ by $\pi$, 
because $\pi\varphi F = \pi F = f\pi$ by the definition of $\pi$.  So, $\varphi F$ lifts $f$ to a finite-sheeted covering and $\varphi F$ satisfies the growth rate inequality. As 
already explained, this implies that $f$ does too. \\
\end{proof}

Given a Thurston map $f:S^2\to S^2$, we will consider branched coverings $\pi:T^2 \to S^2$ such that $\pi^{-1}(P_f)= B_{\pi}$, where $B_{\pi}$ denotes the set of branching points.  Then, restricting $\pi$ to 
$\tilde X= T^2\backslash B_{\pi}$ we obtain a covering space $\tilde X\to X=S^2\backslash P_f$.  In our setting, $\#P_f$ is either $3$ or $4$, then $X$ is homotopy equivalent to the 
wedge sum of $2$ or $3$ circles.  Note that $B_{\pi}$ must be non-empty, and
therefore $\tilde X$ is also homotopy equivalent to a wedge sum of $\#B_{\pi}+1$ circles.  Setting $Y= X\backslash f^{-1}(P_f)$ and 
$\tilde Y=T^2\backslash \pi^{-1}(f^{-1}(P_f))$ we obtain an associated covering space $\tilde Y\to Y$.  This again, modulo homotopy equivalence, is a covering space of the wedge sum 
of $n$ circles, and $n$ is dependent on $\deg(f)$. \\

We will use the  
covering spaces of the wedges of circles illustrated in Figure \ref{tabla}.   We view the wedge sum of $n$ circles corresponding to $X$ (or $Y$) as a graph with one vertex and $n$ edges.  We label the edges and choose orientations for them.  We
use
latin letters for $X$ and greek letters for $Y$.  Each covering space $\tilde X$  is also a graph that we label in such a way that the local picture near each
vertex is the same as in $X$. The covering projection sends all  vertices of $\tilde X$ to the vertex of $X$  and sends each edge of $\tilde X$ to
the edge of $X$ with the same label by a map that is a homeomorphism on the interior of the edge and preserves orientation.  So, for instance, the number of sheets of the covering
space $\tilde X \to X$ is $2,4,6$ and $3$ for the signatures $(2,2,2,2)$, $(2,4,4)$, $(2,3,6)$ and $(3,3,3)$ respectively. Of course analogous considerations can be made for the 
covering spaces $\tilde Y \to Y$.\\

The first two coverings appearing in the picture correspond to the classical square model of the torus, and the last two correspond to the
hexagonal model.  All of the coverings considered are normal. More details are given in each corresponding section.

\begin{figure}
 {\includegraphics[scale = 0.75]{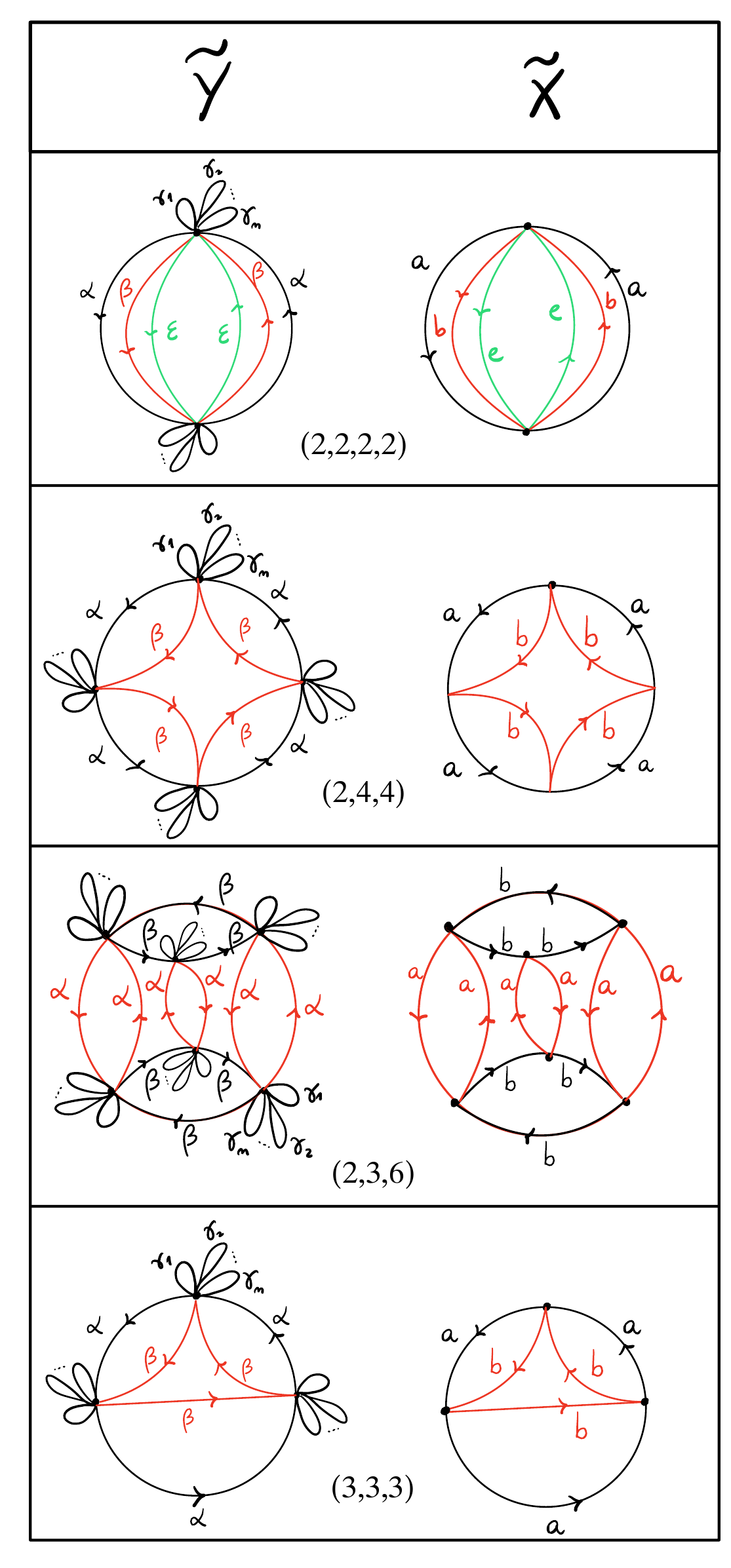}}
\caption{\label{tabla}  }
\end{figure}

\newpage 

\section{$(\bm{2},\bm{2},\bm{2},\bm{2})$}

Let $\Gamma$ be the subgroup of  automorphisms of the complex plane $\C$ generated by $z\mapsto z+1$, $z\mapsto z+i$  and $z\mapsto -z$.  The quotient space $\C/\Gamma$ is
topologically a sphere, 
and the quotient projection $\pi: \C/<z\mapsto z+1, z\mapsto z+i>\to \C/\Gamma$ is a two-fold branched covering of the torus $T^2=\C/<z\mapsto z+1, z\mapsto z+i>$ onto the sphere
with four branched points $b_i, i=1,\ldots,4$.\\

Let $f:S^2\to S^2$ be a branched covering in the class $(2,2,2,2)$, that is, there are four post-critical points $p_1, p_2, p_3, p_4$ which are regular,
 and  such that the local
degree at every  point in their preimage outside $P_f$ is $2$. Assume that $\pi(b_i) = p_i$, $ i=1,\ldots,4$.
Then, $\pi: \tilde X = T^2\backslash \{b_i: i=1,\ldots,4\}\to S^2\backslash P_f =X$ is a normal covering space projection.  Let 
$Y =  X\backslash f^{-1} (P_f)$ and $\tilde Y = T^2\backslash \pi^{-1}(f^{-1} (P_f))$. 

\begin{lemma}  The map $f\pi|_{\tilde Y}:\tilde Y\to X$ lifts to $\tilde X$.
 
\end{lemma}

\begin{proof} 

By Proposition \ref{lift}, it is enought to check that 
$f_*{\pi|_{\tilde Y}}_*(\pi_1 (\tilde Y,\tilde y_0))\subset \pi_* (\pi_1 (\tilde X,\tilde x_0))$. Note that $\pi_1 ( X, x_0)=\left<a,b,e\right>$ where  $a,b$ and $e$ are the homotopy
classes of three loops based at $x_0$
freely homotopic to 
$\sigma_{p_i}$, $i=1,2,3$ respectively. Then,
$$G=\pi_* (\pi_1 (\tilde X,\tilde x_0))=\left<a^2, b^2, e^2, ab, be\right>.\\$$\noindent  We make the identification 
$H_1(X,\Z)\simeq \Z [a]\oplus \Z[b]\oplus \Z[e]$, and write $[\gamma]=m[a] + n [b]+l [e] \in H_1(X,\Z), m,n,l \in \Z$ for $\gamma\in \pi_1(X,x_0) $. Note also 
that
$\gamma \in G$ if and only if $[\gamma]$ verifies $m+n+l =0 \mod 2$ (See Figure \ref{tabla}).\\

Take $y_0=\pi(\tilde y_0)$ and let $\left<\alpha, \beta,\epsilon ,\gamma_1, \ldots, \gamma_n\right>$ generate $\pi_1(Y,y_0)$, where the generators correspond to the punctures at $p_1,p_2,p_3$ and their $n$ 
preimages outside 
$P_f$. That is, each
loop in the generating set is freely homotopic to $\sigma_x$ for some $x\in f^{-1}(P_f)$. Then, 
$$H={\pi|_{\tilde Y}}_*(\pi_1 (\tilde Y,\tilde y_0))=\left < \alpha^2, \beta^2,\alpha\beta,\beta\epsilon, \gamma_i,\alpha \gamma_i\alpha^{-1}, i=1\ldots, n\right>.$$\\
\noindent We make the identification 
$H_1(Y,\Z)\simeq \Z[\alpha]\oplus\Z[\beta]\oplus\Z[\epsilon]\oplus \Z[\gamma_1]\oplus \ldots \oplus \Z[\gamma_n]$, and write
$[\gamma]=m[\alpha]+n[\beta]+ l[\epsilon] + k_1[\gamma_1]+ \ldots+k_n[\gamma_n] \in H_1(Y,\Z), m,n,l, k_i \in \Z, i=1,\ldots,n$ for $\gamma\in \pi_1(Y,y_0) $. Note  
that
$\gamma \in H$ if and only if $[\gamma]$ verifies $m+n+l =0 \mod 2$.\\

  As the points in $P_f$ are regular and $P_f$ is invariant, we get $f_{\#} ([x])=[y]$, where $x=a,b,e$ and $y$ can be either $a, b, e$ or $(abe)^{-1}$ (note that the homology class of 
  a loop in $X$ freely homotopic to $\sigma_{p_4}$ is $-([a]+[b]+[e])$). Besides, as  the critical points are locally $2:1$, we have 
  $f_{\#}([\gamma_i]) = 2[y]$ ( $y$ can be either $a, b, e$ or $(abe)^{-1}$).  Then, 
  $$f_{\#}(m[\alpha]+n[\beta]+ l[\epsilon] + k_1[\gamma_1]+ \ldots+k_n[\gamma_n])=\\$$
  
  $$(m'+\sum_{j:f(c_j)=p_1} 2k_j)[a] + (n'+\sum_{j:f(c_j)=p_2} 2k_j)[b]+\\$$
  
  $$(l'+\sum_{j:f(c_j)=p_3} 2k_j)[e]+(r+\sum_{j:f(c_j)=p_4} 2k_j)([a]+[b]+[e]),$$\noindent where $m+n+l=m'+n'+l'+r$.  It follows that $f_{*}(\gamma)\in G$ if $\gamma \in H$ as we
  wanted.\\

\end{proof}

The previous lemma gives a map $\tilde f: \tilde Y \to 
\tilde X$ such that $\pi \tilde f = f\pi$.

\begin{lemma}\label{ext} The lift $\tilde f:\tilde Y\to \tilde X$ extends continuously to  $F: T^2 \to T^2$ and $\deg(f) = \deg (F)$.
 
\end{lemma}

\begin{proof}  Note that the critical values of $\pi$ are the points $p_i$, $i=1,\ldots,4$. That is, the  $p_i$'s are the only points in $S^2$ that have only one preimage by $\pi$, 
namely the points $b_i\in T^2$, $i=1,\ldots,4$. So, if $ x\in \pi^{-1}(f^{-1}(P_f)$, then $f\pi(x)\in P_f$  and so $f\pi(x)$ has a uniquely defined preimage by $\pi$. We can now define 
$F(x) = \pi^{-1}f\pi(x)$ if $x\notin Y$ and $F(x)= \tilde f(x)$ if $x\in Y$.  The resulting map $F$ is  continuous by construction.
To show that  $\deg(f) = \deg (F)$, just note that we are considering maps  between connected closed orientable two manifolds, and therefore the degree formula holds:
 $\deg(f)\deg(\pi)=\deg(f\pi)=\deg(\pi F)=\deg(\pi)\deg(F)$.\\

\end{proof}

As a consequence, we obtain the following result:

\begin{teo} Let $f:S^2\to S^2$ be a Thurston map with $(2,2,2,2)$ orbifold.  Then $f$ has the rate.
 
\end{teo}

\begin{proof} It is enough to note that we are in the hypothesis of Lemma \ref{accion} with $T(z)=-z$ and  $p: \R^2 \to  \T^2=\R^2/\Z^2$ the standard quotient projection.
 
\end{proof}

In \cite{iprx5} we gave an interesting example in this class:  a map with a $2:1$ invariant indecomposable continuum that is not annular (we
say that a subset of the sphere is annular if it has a neighbourhood homeomorphic to the open annulus). It is not known if it
is possible to have a $2:1$ invariant indecomposable continuum for a rational map of the sphere, and the previously known examples for continuous maps were all annular.  We constructed
the example projecting an Anosov endomorphism in the torus to the sphere via the branched covering projection $\pi: \C/<z\mapsto z+1, z\mapsto z+i>\to \C/\Gamma$, and then performing
a ``derived form pseudo-Anosov'' construction turning  two (previously hyperbolic) fixed points into attractors.  The result is a Wada-lake  totally invariant continuum with
an infinite number of lakes, one for each preimage of the basins of attraction of the fixed points.

\section{$(\bm{2},\bm{4},\bm{4})$} 
Let $\Gamma$ be the subgroup of  automorphisms of the complex plane $\C$ generated by $z\mapsto z+1$, $z\mapsto z+i$  and $z\mapsto iz$.  The quotient space $\C/\Gamma$ is
topologically a sphere, 
and the quotient projection $\pi: \C/<z\mapsto z+1, z\mapsto z+i>\to \C/\Gamma$ is a four-fold branched covering of the torus $T^2=\C/<z\mapsto z+1, z\mapsto z+i>$ onto the sphere
with four branched points $b_i, i=1,\ldots,4$ such that two of them are locally $z\mapsto z^4$ (say $b_1, b_2$) but two of them are locally $z\mapsto z^2$ ($b_3$ and $b_4$). The image of these points under the branched
covering projection is a set of three points; $b_3$ and $b_4$ are mapped to the same point under the branched covering projection.\\

Let $f:S^2\to S^2$ be a branched covering in the class $(2,4,4)$, that is, $P_f= \{d,p,q\}$, $\nu(d)=2$, $\nu(p)=\nu(q)=4$. Note that the set $\{p,q\}$ is $f$-invariant, and 
that $p$ and $q$ are regular points on account that the formula $\nu(f(x))=\nu(x)\deg_x f$ holds for all $x\in S^2$.  Moreover,
if $c$ is critical in $f^{-1}(P_f)$, then $\deg_cf$ is either $2$ or $4$. Assume that 
$\pi(\{b_i: i=1,\ldots,4\}) = \{d, p, q\}$ and that $d=\pi(b_3)=\pi(b_4)$.
Then, $\pi: \tilde X = T^2\backslash \{b_i: i=1,\ldots,4\}\to S^2\backslash \{d, p,q\} = X$ is a normal covering space projection.  Let $Y =   X\backslash f^{-1}(P_f)$
and $\tilde Y = T^2\backslash \pi^{-1}(f^{-1}(P_f))$. 

\begin{lemma}  The map $f\pi|_{\tilde Y}:\tilde Y\to X$ lifts to $\tilde X$.
 
\end{lemma}

\begin{proof} As before, we need to ensure that 
$f_*{\pi|_{\tilde Y}}_*(\pi_1 (\tilde Y,\tilde y_0))\subset \pi_* (\pi_1 (\tilde X,\tilde x_0))$.  Note that $\pi_1 ( X, x_0)=\left<a,b\right>$ 
($a$ and $b$ the homotopy class of two loops based at $x_0$ freely homotopic to $\sigma_p$ and $\sigma_q$ respectively) and
$$G=\pi_* (\pi_1 (\tilde X,\tilde x_0))=\left<a^4, b^4, (ab)^2,a^2ba,a^2b^2\right>.$$\\ \noindent  We make the identification 
$H_1(X,\Z)\simeq \Z [a]\oplus \Z[b]$, and write $[\gamma]=m[a] + n [b] \in H_1(X,\Z), m,n \in \Z$ for $\gamma\in \pi_1(X,x_0) $. Note also 
that
$\gamma \in G$ if and only if $[\gamma]$ verifies $m+n =0 \mod 4$ (See Figure \ref{tabla}).\\

Take $y_0=\pi(\tilde y_0)$ and let $\left<\alpha, \beta,\gamma_1, \ldots, \gamma_n\right>$ generate $\pi_1(Y,y_0)$, where the generators correspond to the punctures at $p,q$ and 
their $n$ 
preimages outside 
$P_f$. That is, each
loop in the generating set is freely homotopic to $\sigma_x$ for some $x\in f^{-1}(P_f)$. Then, 
$$H={\pi|_{\tilde Y}}_*(\pi_1 (\tilde Y,\tilde y_0))=\\$$
$$\left < \alpha^4, \beta^4,(\alpha\beta)^2,\alpha^2\beta\alpha,\alpha^2\beta^2, \gamma_i,\alpha \gamma_i\alpha^{-1}, 
\alpha^2\gamma_i\alpha^{-2}, \alpha^3\gamma_i\alpha^{-3}, i=1\ldots, n\right>.$$\\
\noindent  We make the identification 
$H_1(Y,\Z)\simeq \Z[\alpha]\oplus\Z[\beta]\oplus \Z[\gamma_1]\oplus \ldots \oplus \Z[\gamma_n]$, and write
$[\gamma]=m[\alpha]+n[\beta] + k_1[\gamma_1]+ \ldots+k_n[\gamma_n] \in H_1(Y,\Z), m,n, k_i \in \Z, i=1,\ldots,n$ for $\gamma\in \pi_1(Y,y_0) $. Note  
that
$\gamma \in H$ if and only if $[\gamma]$ verifies $m+n=0 \mod 4$.\\

As the set $\{p,q\}$ is $f$-invariant, and 
 $p$ and $q$ are regular points, we have $f_{\#} ([\alpha])$ and $f_{\#} ([\beta])$ are either $[\alpha]$ or $[\beta]$.   Every critical point $c_i\notin P_f$ that is $2:1$
 satisfy $f(c_i)=d$ with homology class $-([a]+[b])$. Then,
 $f_{\#} ([\gamma_i])=2(-[a]-[b])$, or 
$f_{\#} ([\gamma_i])=4[x]$, where $x$ is either $a$ or $b$.

 Then, 
  $$f_{\#}(m[\alpha]+n[\beta]+ k_1[\gamma_1]+ \ldots+k_n[\gamma_n])=\\$$

  $$(m'+\sum_{j:f(c_j)=p} 4k_j)[a] + (n'+\sum_{j:f(c_j)=q} 4k_j)[b]+(\sum_{j:f(c_j)=d} 2k_j)(-[a]-[b]),$$\noindent where $m+n=m'+n'$. 
  It follows that $f_{*}(\gamma)\in G$ if $\gamma \in H$ as wanted.\\

\end{proof}

By the previous lemma there exists a map $\tilde f: \tilde Y \to \tilde X$ such that $\pi \tilde f = f\pi$.

\begin{lemma}\label{ext2} The lift $\tilde f:\tilde Y\to \tilde X$ extends continuously to  $F: T^2 \to T^2$ and $\deg(f) = \deg (F)$.
 
\end{lemma}

\begin{proof}  Note that for every point $x\in T^2\backslash \tilde Y$ there exists an open neighbourhood $U\subset T^2$ such that 
$\tilde f|_{U\backslash x}: U\backslash x \to \tilde f (U\backslash x)$ is a $k:1$ covering onto its image $V\backslash y$, where $V\subset T^2$ is an open neighbourhood of
the point $y$ and $k\geq 1$.  Defining $F(x)=y$ for $x\in T^2\backslash \tilde Y$ and $F(x)=\tilde f(x)$ on $\tilde Y$ gives a continuous extension. The statement about the degree 
follows as in Lemma \ref{ext}.
 
\end{proof}

As in the previous cases, one has:

\begin{teo} Let $f:S^2\to S^2$ be a Thurston map with $(2,4,4)$ orbifold.  Then $f$ has the rate.
 
\end{teo}

\begin{proof} Exactly as in the previous case we are in the hypothesis of Lemma \ref{accion} with $T(z)=iz$ and  $p: \R^2 \to  \T^2=\R^2/\Z^2$ the standard quotient projection.\\
 
\end{proof}

The last two cases  make use of the hexagonal model of the torus $T^2$.\\

\section{$(\bm{2},\bm{3},\bm{6})$}  
Let $\Gamma$ be the subgroup of  automorphisms of the complex plane $\C$ generated by $z\mapsto z+e^{i\frac{\pi}{3}}$, $z\mapsto z+1$,  
and $z\mapsto e^{i\frac{\pi}{3}}z$.  The quotient space $\C/\Gamma$ is
topologically a sphere, 
and the quotient projection $\pi: \C/<z\mapsto z+e^{i\frac{\pi}{3}}, z\mapsto z+1>\to \C/\Gamma$ is a six-fold branched covering of the torus $T^2=\C/<z\mapsto z+e^{i\frac{\pi}{3}},
z\mapsto z+1>$ onto the sphere
with six branched points $b_i, i=1,\ldots,6$, where $b_1$ is locally $z\mapsto z^6$, $b_2, b_3$ and $b_4$ are locally $z\to z^2$, $b_5$ and $b_6$ are locally $z\to z^3$. \\

Let $f:S^2\to S^2$ be a branched covering in the class $(2,3,6)$, that is,  $P_f=\{p,q,r\}$ such that $\nu(p)=2, \nu(q)=3, \nu(r)=6$. Assume that the image of 
these branched points under the
branched
covering projection is as follows: $\pi(b_1)=p, \pi(b_2)=\pi(b_3)=\pi(b_4)=q, \pi(b_5)= \pi(b_6)=r$. Then, $\pi: \tilde X = \T\backslash \{b_i: i=1,\ldots,6\}\to S^2\backslash \{p,q,r\} = X$ is a normal covering space projection.  Let $Y =   X\backslash f^{-1}(P_f)$
and $\tilde Y = T^2\backslash \pi^{-1}(f^{-1}(P_f))$.

\begin{lemma}  The map $f\pi|_{\tilde Y}:\tilde Y\to X$ lifts to $\tilde X$; that is, there exists $\tilde f:\tilde Y\to \tilde X$ such that $f\pi|_{\tilde Y}=\pi\tilde f$.
 
\end{lemma}

\begin{proof} 

 We need then to ensure that 
$f_*{\pi|_{\tilde Y}}_*(\pi_1 (Y,y_0))\subset \pi_* (\pi_1 (\tilde X,\tilde x_0))$.  Note that $\pi_1 ( X,\tilde x_0)=\left <a,b\right >$ 
($a$ and $b$ the homotopy class of two loops based at $x_0$ freely homotopic to $\sigma_q$ and $\sigma_r$ respectively) and
$$G=\pi_* (\pi_1 (\tilde X,\tilde x_0))=\left <a^2, b^3,a^{-1}b^{-1}ab,ab^3a^{-1},b^{-1}a^{-2}b,b^{-1}ab^{-1}a^{-1}b^2,b^{-2}a^2b^2\right>.$$ \noindent We make the identification 
$H_1(X,\Z)\simeq \Z [a]\oplus \Z[b]$, and write $[\gamma]=m[a] + n [b] \in H_1(X,\Z), m,n \in \Z$ for $\gamma\in \pi_1(X,x_0) $. Note also 
that
$\gamma \in G$ if and only if $[\gamma]$ verifies $3m+2n =0 \mod 6$ (See Figure \ref{tabla}).\\

Take $y_0=\pi(\tilde y_0)$ and let $\left<\alpha, \beta,\gamma_1, \ldots, \gamma_n\right>$ generate $\pi_1(Y,y_0)$, where the generators correspond to the punctures at $p,r$ and 
their $n$ 
preimages outside 
$P_f$. That is, each
loop in the generating set is freely homotopic to $\sigma_x$ for some $x\in f^{-1}(P_f)$. Then, 
$$H={\pi|_{\tilde Y}}_*(\pi_1 (\tilde Y,\tilde y_0))=\\$$

$$ < \alpha^2, \beta^3,\alpha^{-1}\beta^{-1}\alpha\beta,\alpha\beta^3\alpha^{-1},\beta^{-1}\alpha^{-2}\beta,\beta^{-1}\alpha\beta^{-1}\alpha^{-1}\beta^2,\beta^{-2}\alpha^2\beta^2,
$$
$$\gamma_i,\alpha \gamma_i\alpha^{-1}, \alpha\beta^{-1}\gamma_i\beta\alpha^{-1}, \alpha\beta\gamma_i\beta^{-1}\alpha^{-1},\beta\gamma_i\beta^{-1}, \beta^{-1}\gamma_i\beta,
  i=1\ldots, n>.$$\\
\noindent  We make the identification 
$H_1(Y,\Z)\simeq \Z[\alpha]\oplus\Z[\beta]\oplus \Z[\gamma_1]\oplus \ldots \oplus \Z[\gamma_n]$, and write
$[\gamma]=m[\alpha]+n[\beta] + k_1[\gamma_1]+ \ldots+k_n[\gamma_n] \in H_1(Y,\Z), m,n, k_i \in \Z, i=1,\ldots,n$ for $\gamma\in \pi_1(Y,y_0) $. Note  
that
$\gamma \in H$ if and only if $[\gamma]$ verifies $3m+2n=0 \mod 6$.\\

Note that the equation $\deg_xf . \nu(x) = \nu(f(x))$ for all $x\in S^2$  implies  $f(p)=p$, $f(q)$ is either $q$ or $p$, and $f(r)$ is either $r$ or $p$.\\

 Then, 
  $$f_{\#}(m[\alpha]+n[\beta]+ k_1[\gamma_1]+ \ldots+k_n[\gamma_n])=\\$$

  $$(m(1-\delta_{qp})+\sum_{j:f(c_j)=q} 2k_j)[a] + (n(1-\delta_{rp})+\sum_{j:f(c_j)=r} 3k_j)[b]+(3m\delta_{qp}+2n\delta_{rp}+\sum_{j:f(c_j)=p} 6k_j)(-[a]-[b]).$$\noindent 
  
  It follows that $f_{*}(\gamma)\in G$ if $\gamma \in H$ as wanted.\\

\end{proof}

By the previous lemma there exists a map $\tilde f: \tilde Y \to \tilde X$ such that $\pi \tilde f = f\pi$. Moreover, $\tilde f$ extends continuously to  $F: T^2 \to T^2$ and
$\deg (F)=\deg(f)$ (the proof goes exactly as in Lemma \ref{ext2}).\\

As in the previous cases, one has:

\begin{teo} Let $f:S^2\to S^2$ be a Thurston map with $(2,3,6)$ orbifold.  Then $f$ has the rate.
 
\end{teo}

\begin{proof} We are again in the hypothesis of Lemma \ref{accion} with $T(z)=e^{i\frac{\pi}{3}}z$ and  $p: \R^2 \to  \T^2=\C/<z\mapsto z+e^{i\frac{\pi}{3}},
z\mapsto z+1>$ the quotient projection.
 
\end{proof}

\section{$(\bm{3},\bm{3},\bm{3})$} Let $\Gamma$ be the subgroup of  automorphisms of the complex plane $\C$ generated by $z\mapsto z+e^{i\frac{\pi}{3}}$, $z\mapsto z+1$,  
and $z\mapsto e^{i\frac{2\pi}{3}}z$.   The quotient space $\C/\Gamma$ is
topologically a sphere, 
and the quotient projection $\pi: \C/<z\mapsto z+e^{i\frac{\pi}{3}}, z\mapsto z+i>\to \C/\Gamma$ is a three-fold branched covering of the torus $\T=\C/<z\mapsto z+e^{i\frac{\pi}{3}},
z\mapsto z+i>$ onto the sphere
with three branched points $b_i, i=1,\ldots,3$, that are locally $z\to z^3$.\\

Let $f:S^2\to S^2$ be a branched covering in the class $(3,3,3)$: $P_f=\{p,q,r\}$, $\nu(p)=\nu(q)=\nu(r)=3$.  Assume $\pi(b_1)=p, \pi(b_2)=q,\pi(b_3)=r$. 
Then, $\pi: \tilde X = T^2\backslash \{b_i: i=1,\ldots,3\}\to S^2\backslash \{p,q,r\} = X$ is a normal covering space projection.  Let $Y =   X\backslash f^{-1}(P_f)$
and $\tilde Y = T^2\backslash \pi^{-1}(f^{-1}(P_f))$.

\begin{lemma}  The map $f\pi|_{\tilde Y}:\tilde Y\to X$ lifts to $\tilde X$; that is, there exists $\tilde f:\tilde Y\to \tilde X$ such that $f\pi|_{\tilde Y}=\pi\tilde f$.
 
\end{lemma}

\begin{proof} 

 We need then to ensure that 
$f_*{\pi|_{\tilde Y}}_*(\pi_1 (\tilde Y,\tilde y_0))\subset \pi_* (\pi_1 (\tilde X,\tilde x_0))$. Note that $\pi_1 ( X, x_0)=<a,b>$ 
($a$ and $b$ the homotopy class of two loops based at $x_0$ freely homotopic to $\sigma_p$ and $\sigma_q$ respectively)
 and
$G=\pi_* (\pi_1 (\tilde X,\tilde x_0))= <a^3, b^3,(ab)^3,aba>$.

We make the identification 
$H_1(X,\Z)\simeq \Z [a]\oplus \Z[b]$, and write $[\gamma]=m[a] + n [b] \in H_1(X,\Z), m,n \in \Z$ for $\gamma\in \pi_1(X,x_0) $. Note also 
that
$\gamma \in G$ if and only if $[\gamma]$ verifies $m+n =0 \mod 3$ (See Figure \ref{tabla}).\\

Take $y_0=\pi(\tilde y_0)$ and let $\left<\alpha, \beta,\gamma_1, \ldots, \gamma_n\right>$ generate $\pi_1(Y,y_0)$, where the generators correspond to the punctures at $p,q$ and 
their $n$ 
preimages outside 
$P_f$. That is, each
loop in the generating set is freely homotopic to $\sigma_x$ for some $x\in f^{-1}(P_f)$. \\

Then, ${\pi|_{\tilde Y}}_*(\pi_1 (Y,y_0))= \left< a^3, b^3,(ab)^3,aba,\gamma_i,a\gamma_ia^{-1}, a^2\gamma_ia^{-2}, i=1,\ldots,3 \right>.\\$

We make the identification 
$H_1(Y,\Z)\simeq \Z[\alpha]\oplus\Z[\beta]\oplus \Z[\gamma_1]\oplus \ldots \oplus \Z[\gamma_n]$, and write
$[\gamma]=m[\alpha]+n[\beta] + k_1[\gamma_1]+ \ldots+k_n[\gamma_n] \in H_1(Y,\Z), m,n, k_i \in \Z, i=1,\ldots,n$ for $\gamma\in \pi_1(Y,y_0) $. Note  
that
$\gamma \in H$ if and only if $[\gamma]$ verifies $m+n=0 \mod 3$.\\

Note that every point in $P_f$ is regular, and therefore we have $f_{\#}([x]) = [y]$, where $x$ is either $[\alpha]$ or $[\beta]$, and $y$ is either $a$, $b$, or $[-a-b]$. Besides, every 
critical point $c_i \in f^{-1} (P_f)$  is $3:1$, then,
 $f_{\#} ([\gamma_i])=3[x]$ where $x$ is either  $a$, $b$ or $(-[a]-[b])$. Then,

  $$f_{\#}(m[\alpha]+n[\beta]+ k_1[\gamma_1]+ \ldots+k_n[\gamma_n])=\\$$

  $$(m'+\sum_{j:f(c_j)=p} 3k_j)[a] + (n'+\sum_{j:f(c_j)=q} 3k_j)[b]+(l+ \sum_{j:f(c_j)=d} 3k_j)(-[a]-[b]),$$\noindent where $m+n=m'+n'+l$. 
  It follows that $f_{*}(\gamma)\in G$ if $\gamma \in H$ as wanted.\\

\end{proof}

By the previous lemma there exists a map $\tilde f: \tilde Y \to \tilde X$ such that $\pi \tilde f = f\pi$. Moreover, $\tilde f$ extends continuously to  $F: T^2 \to T^2$ and
$\deg (F)=\deg(f)$ (the proof goes exactly as in Lemma \ref{ext2}).\\

As in the previous cases, one has:

\begin{teo} Let $f:S^2\to S^2$ be a Thurston map with $(3,3,3)$ orbifold.  Then $f$ has the rate.
 
\end{teo}

\begin{proof} We are again in the hypothesis of Lemma \ref{accion} with $T(z)= e^{i\frac{2\pi}{3}z}$ and  $p: \R^2 \to  \T^2=\C/<z\mapsto z+e^{i\frac{\pi}{3}},
z\mapsto z+1>$ the quotient projection.\\
 
\end{proof}

\end{document}